\documentclass[a4paper,11pt]{amsart}
\usepackage[latin1]{inputenc}
\usepackage[T1]{fontenc}
\usepackage[english]{babel}
\usepackage{ae,aeguill,icomma}
\usepackage{amssymb}
\usepackage{amsthm}
\usepackage{enumerate}
\usepackage{cleveref}
\usepackage{longtable}

\usepackage{geometry}
\newcommand\N{\mathbb N}

\newcommand\R{\mathbb R}

\DeclareMathOperator\Hess{Hess}
\DeclareMathOperator\conv{conv}
\DeclareMathOperator\inte{int}

\newcommand\ph\varphi
\newcommand\ps\psi
\newcommand\ep\varepsilon
\newcommand\rh\varrho
\newcommand\al\alpha
\newcommand\be\beta
\newcommand\ga\gamma
\newcommand\om\omega
\newcommand\ta\tau
\renewcommand\th\vartheta
\newcommand\de\delta
\newcommand\ze\zeta
\newcommand\ch\chi
\newcommand\et\eta
\newcommand\io\iota
\newcommand\la\lambda
\newcommand\si\sigma
\newcommand\ka\kappa
\newcommand\Ga\Gamma
\newcommand\De\Delta
\newcommand\Th\Theta
\newcommand\La\Lambda
\newcommand\Si\Sigma
\newcommand\Ph\Phi
\newcommand\Ps\Psi
\newcommand\Om\Omega

\newcommand\ov\overline

\theoremstyle{definition}
\newtheorem{satz}{Satz}[section]

\newtheorem{theo}[satz]{Theorem}
\newtheorem{cor}[satz]{Corollary}
\newtheorem{cora}[satz]{Main theorem}

\newtheorem{rem}[satz]{Remark}

\newtheorem{lemma}[satz]{Lemma}
\newtheorem{comi}[satz]{Why are the requirements of the main theorem necessary?}
\newtheorem{comi2}[satz]{What is the purpose of the main theorem (Part 1)?}
\newtheorem{comi3}[satz]{What is the purpose of the main theorem (Part 2)?}
\newtheorem{proposition und beispiel}[satz]{Proposition und Beispiel}
\newtheorem{definition}[satz]{Definition}
\newtheorem{notation}[satz]{Notation}

\newcommand\blfootnote[1]{%
  \begingroup
  \renewcommand\thefootnote{}\footnote{#1}%
  \addtocounter{footnote}{-1}%
  \endgroup
}

\begin{document}
\allowdisplaybreaks
\pagenumbering{arabic}

\title{A new proof for the existence of degree bounds for Putinar's Positivstellensatz}
\author{Tom-Lukas Kriel}
\begin{abstract}
Putinar's Positivstellensatz is a central theorem in real algebraic geometry. It states the following: If you have a set $S= \{ x \in \R^n \ | \ g_1 (x) \geq 0, ... , g_m(x) \geq 0\}$ described by some real polynomials $g_i$, then every real polynomial $f$ that is positive on $S$ can be written as a sum of squares weighted by the $g_i$ and $1$. Consider such an identity $f= \sum_{i=1}^{m} g_i s_i + s_0$. For the applications in polynomial optimization, especially semidefinite programming, the following is important: \\
There exists a bound $N$ for the degrees of the $s_i$ which depends only on the $g_i$, $n$, the degree of $f$, an upper bound for $||f||$ and a lower bound for $\min f(S)$. \\
Two proofs from Prestel and Heß resp. Schweighofer and Nie ([Pr], [He] resp. [Sw], [NS]) for the existence of these degree bounds are known (also for the matrix version of Putinar's Positivstellensatz by Helton and Nie [HN]). Prestel uses valuation and model theory for his approach while Schweighofer gives a constructive solution by using a theorem of P\'{o}lya. \\
In this paper we will give a new elementary, short but non-constructive proof.     
\end{abstract}

\maketitle

\section{Introduction, definitions and historical outline of the problem}

\blfootnote{ 2010 \textit{Mathematics subject classification}: 11E25, 14P10 
\\ \textit{Key words:} Putinar's Positivstellensatz, quadratic module, degree bounds}

\begin{notation} \label{genf}
In general we agree that always $n \in \N=\{1,2,...\}$ and $\ov X=(X_1,...,X_n)^T$ is a tuple of variables. For a commutative ring $R$ with unit $1$ we define the polynomial ring $R[\ov X]$ in $n$ variables and denote by $\deg$ the (total) degree. We will write $R[\ov X]_d:=\{ f \in R[\ov X] \ | \ \deg(f) \leq d \}$ for the set of polynomials up to degree $d$. We extend this notation to matrices with polynomial entries as well by taking the maximum of the degrees of all entries. \\[0.2cm] 
If $m \in \N$ and $g_1,...,g_m \in R[X]$, we write $g=(g_1,...,g_m)$ in shorthand if there is no danger of ambiguity. We define 
\begin{flalign*}
&S_g=\{x \in \R^n \ | \ \forall i \in \{1,...,m\}: g_i(x) \geq 0\} \ \ (\text{semialgebraic set generated by $g$})\\
&P_g=\{f \in \R[\ov X] \ | \ f(S_g) \subseteq \R_{\geq 0} \}. \ \ (\text{positivity cone of $g$}) 
\end{flalign*}
We introduce the following abbreviations, where $A$ is a ring and $C,D \subseteq A$:
\begin{flalign*}
&C+D=\{c+d \ | \ c \in C, \ d \in D\}
\end{flalign*} 
If $a \in A^k$ and $\al \in \N_0^k$ set $a^{\al}=a_1^{\al_1}\cdot...\cdot a_k^{\al_k}$.
We denote the set of symmetric matrices with $S\R^{k \times k}$ and the set of positive semidefinite matrices by $S\R_{\geq 0}^{k \times k}$ (a matrix $A \in \R^{k \times k}$ is said to be positive semidefinite if it is symmetric and $v^T A v \geq 0$ for all $v \in \R^k$. For $A,B \in S\R^{k \times k}$ we write $A \succ B$ if $A-B$ is positive definite.  \\[0.2cm] 
We often will treat finite-dimensional $\R$-vector spaces as Banach spaces. More precisely, we can introduce a norm by giving an isomorphism to a power of $\R^k$ and pulling back a standard norm. Since all norms on $\R^k$ are equivalent we get a unique topology.
\end{notation}

\begin{definition} \label{quadratic}
Let $k \in \ N$, $A=\R[\ov X]^{k \times k}$ and $g_1,...,g_m \in \R[\ov X]$. We define the \textit{quadratic module} $M_g$ generated by the $g_i$ (in $A$) as
\begin{align*}
&M_g= \{ S_0 + g_1S_1 + ... +g_mS_m \ | \ S_0,...,S_m \in \text{SOS}(A)\} \ \ \text{, where} \\
&\text{SOS}(A)=\left\{C^T C \ \middle| \ h \in \N_0, \ C \in \R[\ov X]^{h \times k} \right\} \ \ \ \ (\text{\textit{sums of squares}})
\end{align*}
Notice that in the case $k=1$ the set $\text{SOS}(A)$ consists of sums of squares of elements of $A$. Obviously, if one evaluates a sum of squares in a point, the result is non-negative. For $k>2$ squares of quadratic matrices of size $k$ do not share a similar property. However an element of $\text{SOS}(A)$ evaluated in a point is positive semidefinite. Therefore the name \textit{sums of squares} is a bit misleading but standard in real algebraic geometry. \\[0.2cm]
Quadratic modules can be defined in a far more general setting, which we do not need to do for our purposes. Quadratic modules are closed under addition and multiplication with sums of squares. If they are also closed under multiplication, they are called \textit{preorderings}. The preordering defined by the $g_i$ (in $A$) will be denoted by $T_g$ and we have:
\begin{align*}
T_g=M_{(g^\al \ | \ \al \in \{0,1\}^m)}=\left\{ \sum_{\al \in \{0,1\}^m} g^\al s_{\al} \ \middle| \ s_\al \in \sum A^2 \right\}
\end{align*} 
For $N \in \N$ we define the \textit{truncated quadratic module} (resp. \textit{preordering}): 
\begin{align*}
&M_{g,A}[N]= \{ S_0 + g_1S_1 + ... +g_mS_m \ | \ S_0,...,S_m \in \text{SOS}(A), \ \deg(S_0) \leq N, \deg(S_i g_i) \leq N \} \\
&T_{g,A}[N]=\left\{ \sum_{\al \in \{0,1\}^m}  g^\al S_\alpha \ \middle| \ S_\al \in \text{SOS}(A), \ \deg(S_{\al} g^\al) \leq N \text{ for all } \al \in  \{0,1\}^m\right\}
\end{align*}
$M_g$ is called \textit{Archimedean} if for every $a \in A$ there exists $N \in \N$ such that $a+N \in M_g$. It is a well-known fact that $M_g$ is Archimedean iff there exists $N \in \N$ such that $N-\sum_{i=1}^{n}X_i^2 \in M_g$ ([Ma], 5.2.4).
\end{definition}

\begin{rem}
It is easy to see that $M_g[N] \subseteq T_g[N] \subseteq P_g$ and $M_g=\bigcup_{N \in \N} M_g[N], \ T_g=\bigcup_{N \in \N} T_g[N] $ for $N \in \N$. \\[0.2cm]
Much effort in real algebraic geometric has been invested to examine how small the difference between $P_g$ and $M_g$ (resp. $T_g$) is. This is due to the fact that one wants to check if a polynomial $f$ is contained $P_g$ e.g. for solving polynomial optimization problems. However this is difficult to verify with a computer because it is difficult to determine how $S_g$ looks like. Contrary to that, it is possible to check whether $f \in M_g[N]$ in most cases with a computer (see [Lau], Section 3.3). Also for theoretical aspects this connection is quite important. \\[0.2cm] 
The following two theorems are the most important and fundamental statements regarding this question. 
\end{rem}

\begin{theo} \label{putinar} (Putinar's Positivstellensatz) ([Pu], 1993) Let $g_1,...,g_m \in \R[\ov X]$ be polynomials defining an Archimedean module $M_g$. Then we have $f \in M_g$ for every $f \in \R[\ov X]$ satisfying $f > 0$ on $S_g$. 
\end{theo}

\noindent Putinars Positivstellensatz was discovered shortly after Schmüdgens Positivstellensatz which is a similar statement having a weaker hypothesis and a weaker conclusion:

\begin{theo} \label{schmudgen} (Schmüdgen's Positivstellensatz) ([Sm], 1991) Let $g_1,...,g_m \in \R[\ov X]$ be polynomials defining a compact set $S_g$. Then we have $f \in T_g$ for every $f \in \R[\ov X]$ satisfying $f > 0$ on $S_g$. 
\end{theo}

\noindent See also Marshalls book [Ma] for a good presentation of the proofs. One can generalize these theorem to matrices in order to get the following statements:

\begin{theo} \label{putinarmatrix} (Putinar's Positivstellensatz for matrix polynomials) ([HS], 2006) Let $g_1,...,g_m \in \R[\ov X]$ be polynomials defining an Archimedean module $M_{g,\R[\ov X]}$ and $k \in \N$. Then we have $A \in M_{g,\R[\ov X]^{k \times k}}$ for every $A \in S\R[\ov X]^{k \times k}$ satisfying $A \succ 0$ on $S_g$. 
\end{theo}

\begin{theo} \label{schmudgenmatrix} (Schmüdgen's Positivstellensatz for matrix polynomials) ([He], 2013) Let $g_1,...,g_m \in \R[\ov X]$ be polynomials defining a compact set $S_g$ and $k \in \N$. Then we have $A \in T_{g,\R[\ov X]^{k \times k}}$ for every $A \in S\R[\ov X]^{k \times k}$ satisfying $A \succ 0$ on $S_g$. 
\end{theo}

\noindent In order to prove his theorem Schmüdgen showed that $T_g$ is Archimedean iff $S_g$ is compact, which was the hardest part. The previous theorems give information when a polynomial or matrix polynomial is in a quadratic module. We are interested in the question if one is able to control the degree of the summands in the weighted sums-of-squares representation. Our aim is to prove the following (which will be done in Chapter 2):\\

\noindent{\bfseries Main Theorem 3.2.} (Putinar's Positivstellensatz for matrix polynomials with degree bounds) Let $k \in \N$ and $g_1,...,g_m \in \R[\ov X]$ be polynomials defining an Archimedean quadratic module $M_{g,\R[\ov X]}$. Then for all $L \in \N$ there exists $N \in \N$ guaranteeing that for every $A \in S\R[\ov X]_L^{k \times k}$ satisfying $||A|| \leq L$ and $A \succeq \frac{1}{L}$ on $S_g$ already $A \in M_{g,\R[\ov X]^{k \times k}}[N]$ holds.

\begin{rem}
There already exist two other proofs for the main theorem. The first proof was given by Prestel ([Pr]); later Schweighofer gave another proof ([Sw]) (their proofs dealt only with the scalar case and with the Schmüdgen-setting instead of the Putinar-setting but their approaches were generalized by [He] resp. [NS] and [HN] to the matrix case in both settings). \\[0.2cm] 
Prestel's strategy was to show a version Schmüdgen's Positivstellensatz holding over arbitrary real closed fields. Then he used the Finitess theorem from model theory to conclude the existence of degree bounds, which is nowadays a common technique. \\[0.2cm] 
Schweighofer gave a constructive proof for Schmüdgen's Positivstellensatz using a theorem of P\'{o}lya where he was able to analyze the degrees appearing. \\[0.2cm] 
Our proof is by far the shortest one and the most elementary. However we use Schmüdgen's Positivstellensatz without degree bounds as a starting point whereas Prestel and Schweighofer did not need to do that and reproved this.  
\end{rem}

\section{The purpose of the main theorem and its requirements}

\noindent In this Section we will argue why some requirements in the main theorem are optimal and what the purpose of this theorem is. Since we do not need this Section for our proof, it is possible to omit it or read it after the proof.

\begin{comi}
Scheiderer showed that $T_g \neq P_g$ if the dimension of $S_g$ as a semialgebraic set is at least 3 (see [Ma], 2.6.2 for a proof) so the previous theorems are not valid anymore if one requires only that $f \geq 0$ on $S_g$ (resp. $A \succeq 0$ on $S_g$). If $T_g \cap -T_g = \emptyset$, (this is fulfilled if $0 \notin g$ and $S_g$ has an interior point) one can show that $T_g[N]$ is closed ([Ma], 4.1.4). Combining this fact with the result of Scheiderer one sees that both theorems do not allow the existence of degree bounds in the following sense: \\[0.2cm] 
If $S_g$ has dimension at least 3, there is an $K \in \N$ such that $\{f \in \R[\ov X]_K \ | \ f \geq 0 \text{ on } S_g \} \nsubseteq T_g[N]$ for all $N \in \N$. So this means that one cannot expect the existence of degree bounds in the representation by just bounding the degree of the polynomials $f$. An explicite example is that for $g=X^3(1-X)$ there exists no $N \in \N$ such that 
\begin{align*}
\{X+\epsilon \ | \ \epsilon > 0\} \subseteq M_g[N]
\end{align*}  
because $X \notin T_g$. $M_g=T_g$ is even Archimedean in this case. \\[0.2cm] 
Therefore our degree bound will also depend on how small our functions become on $S_g$. Of course if we take a function $f$ to represent, then it does not change anything if we multiply it with a positive scalar. So in order to limit the minimum of our functions we will require that we have an upper bound on the norm in $\R[X]_d$ and a lower bound for $\min f(S)$.
\end{comi}

\begin{comi2}
The most important application of the main theorem is to determine whether the so called Lasserre relaxation becomes exact. For some polynomials $g=(g_1,...,g_m) \subseteq \R[\ov X]$ the Lasserre relaxation is a sequence of descending super sets $(\mathcal{L}_{{g,k}})_{k \in \N}$ of $S_g$ which are designed to approximate $S_g$. The sets $\mathcal{L}_{{g,k}}$ are (affine) projections of spectrahedra. A spectrahedron is the preimage of $S\R_{\geq 0}^{k \times k}$ under an affine linear-map $\R^l \rightarrow S\R^{k \times k}$.
Projections of spectrahedra are nice to handle from a computational point of view because the distance between them and hyperplanes can be computated under mild assumptions quite efficiently (as a generalization of linear programming). This task is called \textit{semidefinite programming}. So a semidefinite program is of the form
\begin{align*}
\text{minimize } \ell(x) \text{ where } x \in P
\end{align*}
where $\ell \in \R[\ov X]_1$ and $P$ is a spectrahedron (by introducing additional variables it is also allowed that $P$ is only a projection of a spectrahedron). Many practical optimization problems can be approximated by semidefinite programs. See [Ma] and [GM] for more information about semidefinite programming.\\[0.2cm] 
The Lasserre relaxation introduced in [Las] can be seen as an attempt to find the minimum of a linear polynomial on a basic-closed semialgebraic set $S_g$ by approximating $S_g$ with $\mathcal{L}_{{g,k}}$ and solving the related semidefinite programs afterwards. \\[0.2cm] 
Lasserre showed that if one assumes that $M_g$ is Archimedean and $S_g$ convex, then $(\mathcal{L}_{{g,k}})_{k \in \N}$ converges to $S_g$ in a strong way (for given $\ep > 0$ we find $k \in \N$ such that $\mathcal{L}_{{g,k}} \subseteq S_g + B(0,\ep)$). Although it is not necessary to argue with degree bounds, Lasserre used the \Cref{putinarmatrixbound} for $k=1$ in his proof ([Las], Theorem 6). Of course in order to get an exact approximation one would like to have $\mathcal{L}_{{g,k}}=\conv(S_g)$.
This question is strongly connected to the topic of this paper via the following lemma:    
\end{comi2}

\begin{lemma} \label{lemmi} ([NPS], 3.1)
Let $g=(g_1,...,g_m) \subseteq \R[\ov X]$, $S_g$ be with non-empty interior and $M_g$ Archimedean. Then for $k \in \N$ the following is equivalent:
\begin{flalign*}
&(i) \ \ \conv(S_g)=\mathcal{L}_{{g,k}} \\
&(ii) \ M_g[k] \cap \R[\ov X]_1 = \{ \ell \in \R[\ov X]_1 \ | \ \ell(S_g) \subseteq \R_{\geq 0} \} 
\end{flalign*}
\end{lemma}

\begin{comi3}
So the basic strategy to see whether the approximation becomes exact is to verify that $(ii)$ is fulfilled instead of checking $(i)$. One cannot apply the main theorem directly because in (ii) one has no control about the minimality of the polynomials. However in the case that all $g_i$ have negative definite Hessian on $S_g$ (this is a strong assumption but a rather natural one) one can argue in the following way which was done by Helton and Nie in [HN] and in a weaker form by Lasserre earlier in [Las] (improvements of the following strategy work in more general cases): \\[0.2cm] 
It is possible to use the Karush-Kuhn-Tucker result from optimization in order to get Lagrange multipliers $\lambda_0,...,\lambda_m \geq 0$ and $u \in S_g$ depending on $\ell \in \R[\ov X]_1$ with $\ell(S_g) \subseteq \R_{\geq 0}$ which satisfy the following equality:
\begin{align*}
\ell=\lambda_0 + \sum_i^m \lambda_i g_i + \sum_i^m \lambda_i (\ov{X}-u)^T \underbrace{\int_0^1 \int_0^t (-\Hess g_i)(u+s(\ov{X}-u)) \ ds \ dt}_{=:H_{i,\ell}} (\ov{X}-u)
\end{align*}
Now after making some basic calculations one is able to apply the main theorem and conclude that there is a uniform $N \in \N$ guaranteeing $H_{i,\ell} \subseteq M_{g,\R[ \ov X]^{n \times n}}[N]$. From the upper equation it follows
\begin{align*}
\ M_g[N+2] \cap \R[\ov X]_1 = \{ \ell \in \R[\ov X]_1 \ | \ \ell(S_g) \subseteq \R_{\geq 0} \}
\end{align*}
so $(ii)$ of \Cref{lemmi} is fulfilled. \\[0.2cm] 
For more information about the Lasserre relaxation we refer the reader to [Las] and to [HN] for the details of the results about the exactness of the Lasserre relaxation where the main theorem is used.
\end{comi3}

\section{Proof of the theorem}

\begin{lemma} \label{u1}
Let $k \in \N$, $R=\R[\ov X]^{k \times k}$, $g_1,...,g_m \in \R[\ov X]$ be polynomials defining the quadratic module $M_{g,R}$, $H$ a finite-dimensional subspace of $R$ and $U \subseteq M_{g,R} \cap H$ a compact set in $H$ consisting of inner points of $M_{g,R} \cap H$ in $H$ (i.e. $U \subseteq \inte(M_{g,R} \cap H)$ in $H$). Then there is $N \in \N$ with $U \subseteq M_{g,R}[N]$.
\end{lemma}

\begin{proof} 
We work in the topology of $H$. Because of the compactness of $U$, one only has to prove $U \subseteq \bigcup_{N \in \N} \inte(M_{g,R}[N] \cap H)$. So take $f \in U$. Choose a basis $v_1,...,v_k$ of $H$. As $f$ is an inner point of $M_{g,R} \cap H$ in $H$, there is $\ep > 0$ such that 
\begin{align*}
G_f:= \{ f + \ep v_1, ... , f + \ep v_k, f - \ep v_1, ... , f - \ep v_k\} \subseteq M_{g,R}.
\end{align*}
Because $G_f$ is finite, we can find $N \in \N$ with $G_f \subseteq M_{g,R}[N]$. Due to the convexity of $M_{g,R}[N]$ it also is true that $\conv(G_f) \subseteq M_{g,R}[N]$. Now $\conv(G_f)$ is a neighbourhood of $f$ in $H$, hence $f \in \inte(M_{g,R}[N] \cap H)$.  
\end{proof}

\begin{cora} \label{putinarmatrixbound} (Putinar's Positivstellensatz for matrix polynomials with degree bounds) Let $k \in \N$ and $g_1,...,g_m \in \R[\ov X]$ be polynomials defining an Archimedean quadratic module $M_{g,\R[\ov X]}$. Then for all $L \in \N$ there exists $N \in \N$ guaranteeing that for every $A \in S\R[\ov X]_L^{k \times k}$ satisfying $||A|| \leq L$ and $A \succeq \frac{1}{L}$ on $S_g$ already $A \in M_{g,\R[\ov X]^{k \times k}}[N]$ holds.
\end{cora}

\begin{proof}
Define $H=\R[\ov X]_L^{k \times k}$, $R=\R[\ov X]^{k \times k}$ as well as
\begin{align*}
U=\left\{ A \in \R[\ov X]_L^{k \times k} \ \middle| \ ||A|| \leq L, A \succeq \frac{1}{L} \text{ on } S_g \right\}
\end{align*}
and apply \Cref{u1} with these data as an input. Putinar's Positivstellensatz (without degree bounds) \Cref{putinarmatrix} guarantees that $U \subseteq \inte(M_{g,R} \cap H)$ in $H$. Of course $U$ is compact (use that the map sending a symmetric matrix to its smallest eigenvalue w.r.t the modulus is continuous) so all the preliminaries of the \Cref{u1} are fulfilled.
\end{proof}

\begin{cor} \label{putinarbound} (Putinar's Positivstellensatz with degree bounds)
Let $g_1,...,g_m \in \R[\ov X]$ be polynomials defining an Archimedean quadratic module $M_g$. Then for all $L \in \N$ there exists some $N \in \N$ guaranteeing that for every $f \in \R[\ov X]_L$ satisfying $||f|| \leq L$ and $f \geq \frac{1}{L}$ on $S_g$ already $f \in M_g[N]$ holds.
\end{cor}

\begin{proof} 
Apply \Cref{putinarmatrixbound} with $k=1$.
\end{proof}

\begin{cor} \label{schmudgenmatrixbound} (Schmüdgen's matrix Positivstellensatz with degree bounds) Let $k \in \N$ and $g_1,...,g_m \in \R[\ov X]$ be polynomials defining a compact set $S_g$. Then for al $L \in \N$ there exists $N \in \N$ guaranteeing that for every $A \in S\R[\ov X]_L^{k \times k}$ satisfying $||A|| \leq L$ and $A \succeq \frac{1}{L}$ on $S_g$ already $A \in T_{g,\R[\ov X]^{k \times k}}[N]$ holds.
\end{cor}

\begin{proof}
The proof runs analogously to the proof of \Cref{putinarmatrixbound}. Alternatively one can use that $T_g$ is Archimedean and apply \Cref{putinarmatrixbound} directly.
\end{proof}

\begin{cor} \label{schmudgenbound} (Schmüdgen's Positivstellensatz with degree bounds)
Let $g_1,...,g_m \in \R[\ov X]$ be polynomials defining a compact set $S_g$. For all $L \in \N$ there exists some $N \in \N$ guaranteeing that for every $f \in \R[\ov X]_L$ satisfying $||f|| \leq L$ and $f \geq \frac{1}{L}$ on $S_g$ already $f \in T_g[N]$ holds.
\end{cor}

\begin{proof} 
Apply \Cref{schmudgenmatrixbound} with $k=1$.
\end{proof}

\section{references}

\begin{small}
\begin{longtable}{p{1cm} p{12.8cm}}
$[$GM$]$ & \textbf{B. Gärtner, J. Matousek}: {\it Approximation Algorithms and Semidefinite Programming}, Springer, 2012 \\
$[$He$]$ & \textbf{R. Heß}: {\it Die Sätze von Putinar und Schmüdgen für Matrixpolynome mit Gradschranken}, unpublished diploma thesis (in german), University of Konstanz, supervised by M. Schweighofer, 2013. \\ 
$[$HN$]$ & \textbf{J.W. Helton, J. Nie}: {\it Semidefinite representation of convex sets },
Math. Program., 122 (1), p. 21-64, 2010. \\
$[$HS$]$ & \textbf{C.W.J. Hol, C.W. Scherer}: {\it Matrix sum-of-squares relaxations for robust semi-definite
programs}, Math. Program. 107, no. 1-2, Ser. B, p.189-211, 2006. \\
$[$Lau$]$ & \textbf{M. Laurent}: {\it Sums of squares, moment matrices and optimization over polynomials}, Emerging Applications of Algebraic Geometry, Vol. 149 of IMA Volumes in Mathematics and its Applications, Springer, p. 157-270, 2009. \\
$[$Las$]$ & \textbf{J.B. Lasserre}: {\it Convex sets with semidefinite representation}, \newline Math. Program., 120, p.457-477, 2009. \\
$[$Ma$]$ & \textbf{M. Marshall}: {\it Positive polynomials and sums of squares}, \newline American Mathematical Society, 2008. \\
$[$NPS$]$ & \textbf{T. Netzer, D. Plaumann, M. Schweighofer}: {\it Exposed faces of semidefinitely representable sets},
SIAM J. Optimization, 20(4), p.1944-1955, 2010. \\
$[$NS$]$ & \textbf{J. Nie, M. Schweighofer}: {\it On the complexity of Putinar's Positivstellensatz}, Journal of Complexity 23, No 1, p. 135-150, 2007. \\
$[$Pr$]$ & \textbf{A. Prestel}: {\it Bounds for Representations of Polynomials Positive on Compact Semi-Algebraic Sets}, in F.-V. Kuhlmann, S. Kuhlmann, M. Marshall (editors): Valuation Theory and its Applications I, Valuation Theory and its Applications I, 2001. \\
$[$Pu$]$ & \textbf{M. Putinar}: {\it Positive polynomials on compact semi-algebraic sets}, Indiana University Math. Journal 42, No 3, p. 969-984, 1993 \\
$[$Sm$]$ & \textbf{K. Schmüdgen}: {\it The K-moment problem for compact semi-algebraic sets}, Math. Ann. 289, No 2, p. 203-206, 1991 \\
$[$Sw$]$ & \textbf{M. Schweighofer}: {\it On the complexity of Schmüdgen's Positivstellensatz}, Journal of Complexity 20, No 4, p. 529-543, 2004.
\end{longtable}
\end{small}

\end{document}